\documentclass[11pt]{article}
\usepackage[a4paper,vmargin={3.2cm,3.2cm},hmargin={1.9cm,1.9cm}]{geometry}
\usepackage{setspace}
\usepackage[utf8]{inputenc}
\usepackage[english]{babel}
\usepackage{graphicx}
\usepackage{multicol}
\usepackage{amsmath,amsfonts,amssymb, amsthm}
\usepackage{float}
\usepackage{bbm}
\usepackage{mathrsfs} 
\usepackage[version=0.96]{pgf}
\usepackage{graphicx}
\usepackage[final]{pdfpages}
\usepackage{enumitem}
\usepackage{multirow}
\usepackage{fancyvrb}
\usepackage[pdftex, colorlinks, linkcolor=blue,citecolor=red]{hyperref}
\usepackage{tikz}
\usepackage{ dsfont }
\usepackage{graphicx}
\usepackage{wasysym}
\usepackage{caption}
\usepackage{subcaption}
\usepackage{fontawesome}
\usetikzlibrary{arrows,automata}
\theoremstyle{plain}
\newtheorem{theorem}{Theorem}
\newtheorem{lemma}[theorem]{Lemma}
\newtheorem{prop}[theorem]{Proposition}

\newtheorem{corollary}[theorem]{Corollary}

\renewcommand{\leq}{\leqslant}
\renewcommand{\geq}{\geqslant}

\renewcommand{\P}{\mathbb{P}}
\newcommand{\E}{\mathbb{E}}

\newcommand{\Z}{\mathbb{Z}}

\newcommand{\Top}{\operatorname{Top}}

\newcommand{\vphi}{\varphi}
\newcommand{\eps}{\varepsilon}

\newcommand{\cT}{\mathcal{T}}

\newcommand{\tvphi}{\tilde{\varphi}}

\newcommand{\tS}{\tilde{S}}

\newcommand{\one}{\mathds{1}}
\newcommand{\beq}{\begin{eqnarray*}}
\newcommand{\eeq}{\end{eqnarray*}}

\title{Sharpness of the phase transition for parking on random trees}

\author{Alice Contat\footnote{Universit\'e Paris-Saclay E-mail: \texttt{alice.contat@universite-paris-saclay.fr} }}
\date{ }

\begin{document}

\maketitle
\begin{abstract}
Recently, a phase transition phenomenon has been established for parking on random trees in \cite{CG19, GP19,jones19,LP16,P20}. We extend the results of \cite{CH19} on general Galton--Watson trees and allow different car arrival distributions depending on the vertex outdegrees. We then prove that this phase transition is \emph{sharp} by establishing a large deviations result for the flux of exiting cars. This has consequences on the offcritical geometry of clusters of parked spots which displays similarities with the classical Erd\H{o}s-Renyi random graph model.

\begin{figure}[h!]
\centering
\begin{tikzpicture}[ xscale = 0.7, yscale = 0.95, node distance=2cm]
\tikzstyle{every state}=[fill=black!10,draw=none, text=red]
\node[state] at (0,0)  (P1) {};
\node[state] at (-1,2) (P2)   {};
\node[state] at (1,2) (P3)  {};
\node[state] at (1,4) (P4) {};
\node[state] at (3,4) (P5)  {};
\node[state] at (5,4) (P6)  {};
\node[state] at (0,6) (P7)   {};
\node[state] at (2,6) (P8)   {};
\node[state] at (5,6) (P9)   {};
\node[state] at (-1,4) (P10)   {};
\node[state] at (-3,4) (P11)   {};
\draw[->]  (P1) -- (0,-2);
\draw (-2,2) node {\faCar};
\draw (-2.7,2) node {\faCar};
\draw (-3.4,2) node {\faCar};


\draw (0,6.75) node {\faCar};

\draw (3,4.75) node {\faCar};
\draw (3,5.15) node {\faCar};

\draw (5,6.75) node {\faCar};
\draw (5,7.15) node {\faCar};

\draw (6,4) node {\faCar};
\path[->, >=stealth, shorten >=1pt]
(P2) edge (P1)
(P3) edge (P1)
(P4) edge (P3)
(P5) edge (P3)
(P6) edge (P3)
(P7) edge (P4)
(P8) edge (P4)
(P9) edge (P6)
(P10) edge (P2)
(P11) edge (P2);

\end{tikzpicture}
\qquad
\begin{tikzpicture}[ xscale = 0.7, yscale = 0.95, node distance=1cm]

\tikzstyle{every state}=[fill=black!10,draw=none, text=blue]
\node[state] at (0,0)  (P1) {\faCar};
\node[state] at (-1,2) (P2)   {\faCar};
\node[state] at (1,2) (P3)  {\faCar};
\node[state] at (1,4) (P4) {};
\node[state] at (3,4) (P5)  {\faCar};
\node[state] at (5,4) (P6)  {\faCar};
\node[state] at (0,6) (P7)   {\faCar};
\node[state] at (2,6) (P8)   {};
\node[state] at (5,6) (P9)   {\faCar};
\node[state] at (-1,4) (P10)   {};
\node[state] at (-3,4) (P11)   {};
\draw[->]  (P1) -- (0,-2);
\path[->, >=stealth, shorten >=1pt]
(P2) edge [ultra thick] node[left] {2} (P1) 
(P3) edge [ultra thick] node[right] {1} (P1)
(P4) edge  (P3)
(P5) edge [ultra thick] node[left, near start] {1}(P3)
(P6) edge [ultra thick] node[below right] {1}(P3)
(P7) edge (P4)
(P8) edge (P4)
(P9) edge [ultra thick] node[left] {1}(P6)
(P10) edge (P2)
(P11) edge (P2)
;
\draw (0.5,-1.75) node {\faCar};
\draw (1.2,-1.75) node {\faCar};

\end{tikzpicture}
\caption{Illustration of the parking process of $9$ cars on a tree. On the left: a rooted tree together with a configuration of cars trying to park. On the right, the resulting parking configuration with flux on the edges and with two cars which did not manage to park on the tree.}\label{parking}
\end{figure}
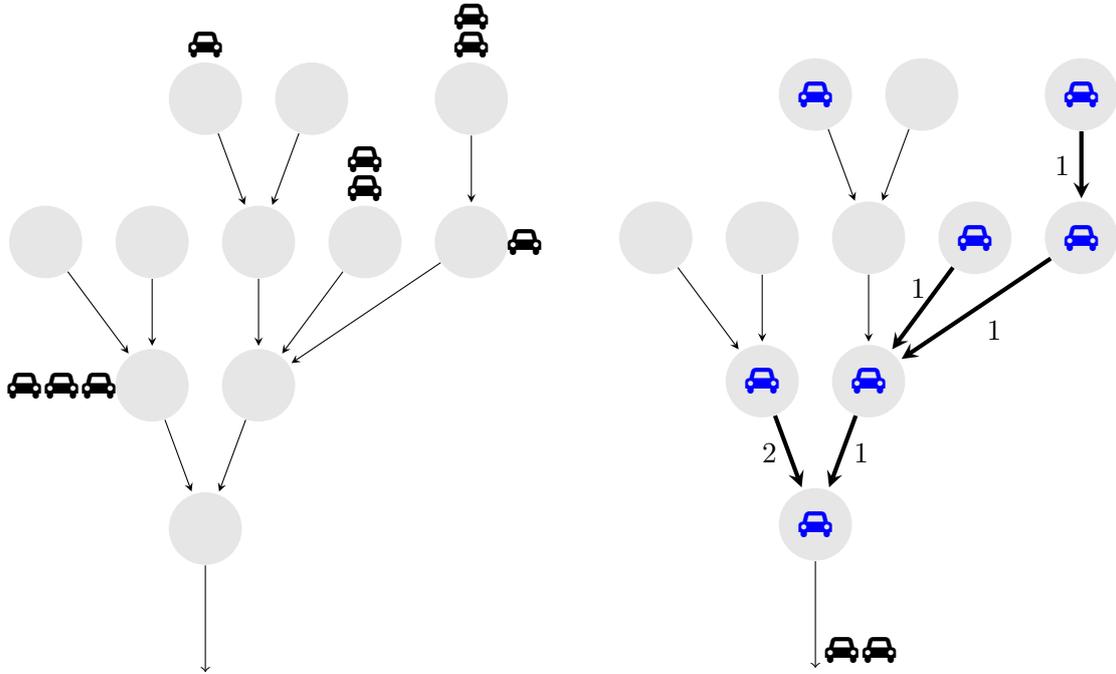

\end{abstract}

\newpage
 
\section{Introduction}

Parking functions on the line are combinatorial objects first introduced by Konheim and Weiss in \cite{KW66} in the context of collision in hashing functions.  Since then, many generalizations of the parking procedure have been studied, most notably on plane trees. On critical Galton--Watson trees with i.i.d.~car arrivals on the vertices, Curien and Hénard proved in \cite{CH19} that  the parking procedure undergoes a phase transition: when the ``density” of cars is small, then the probability that all cars can park is large, whereas when the density is too large, then with high probability, there is at least one car that will not manage to park. This transition was first observed by Lackner and Panholzer  \cite{LP16}, then by Goldschmidt and Przykucki  \cite{GP19} on Cayley trees with cars arriving uniformly on the vertices. Other particular cases have been studied in \cite{CG19,jones19,P20}. The phase transition was also proved in related models: see \cite{BBJ19, GP19} for the case of supercritical Galton--Watson trees and \cite{CD20} for a similar framework on regular trees.

In this work, we generalize the results of \cite{CH19} by allowing the distributions of car arrivals to depend on the vertex's outdegree. But most importantly we show that this phase transition is \emph{sharp}. Establishing sharpness of phase transition in statistical mecanics models is a crucial step in the understanding of the transition and in particular offcritical regimes, see \cite{AB87,DC19}. In our case, this sharpness will appear as large deviations for the flux of cars in the subcritical and supercritical cases (see Theorem \ref{GrandDev}) and will have direct consequences on the geometry of clusters in these regimes (Corollary \ref{CorGeom}).
\paragraph{General Phase Transition.}
Let us first recall the parking procedure on a rooted tree $\mathbf{t}$ i.e.~a tree with a distinguished vertex called \emph{the root} and denoted by $\diameter$. We assume that the edges of $\mathbf{t}$ are oriented towards the root and consider the vertices of $\mathbf{t}$ as parking spots. Imagine now that cars arrive one after the other on the vertices of $\mathbf{t}$. Each car tries to park at its arriving parking spot. If the parking spot is empty, the car stops there. If not, the driver follows the edges towards the root and takes the first available space, if there is one. If not, the car leaves without parking (see Figure \ref{parking}). An important property of this model is its Abelian property: changing the order of the car arrivals does not affect the final configuration and the number of cars that exit the tree.

We consider here a slightly more general model than in \cite{CH19} by allowing the law of car arrivals to depend on the outdegree of the vertex. Specifically, given a rooted tree $\mathbf{t}$, we suppose that the arrivals of the cars on each vertex of $\mathbf{t}$ are independent random variables and that their law only depends on the outdegree of the vertex i.e.~its number of children. In what follows, we simply write degree instead of outdegree (note the special role of the root vertex).  We denote by $(L_x : x \in \mathbf{t})$ the car arrivals on the vertices of $\mathbf{t}$. The common law of the arrival of cars on a vertex of degree $k$ is 
\begin{equation}
\mu_{(k)} \text{ with mean } m_{(k)} \geq 0 \text{ and finite variance }\sigma_{(k)}^{2}.
\end{equation}
In this paper we shall only deal with (rooted) \emph{plane} trees (i.e.~such that the children of a given vertex are ranked from left to right), which are versions of critical Galton--Watson tree with offspring distribution
\begin{equation}
\nu  = \sum_{k = 0}^{\infty} \nu_k \delta_k \text{ with mean } 1 \text{ and finite variance }\Sigma^2,
\end{equation}
the classical Galton--Watson tree $\cT$, the Galton--Watson tree conditioned to have $n$ vertices $\cT_n$\footnote{In all this paper, we shall implicitly restrict to the values of $n$ for which $\P (|\cT| =n) >0$.} and the Galton--Watson tree conditioned to survive forever $\cT_{\infty}$.
We assume throughout the paper that the number of cars arriving on a ``typical" vertex has exponential tail, i.e.~
\begin{equation}\label{Hexp}\tag{$H_{\exp}$} 
F(z) = \sum_{k \geq 0} \nu_{k} \sum_{i \geq 0} \mu_{(k)} (i) z^i 
\end{equation} 
has a radius of convergence strictly larger than $1$. We also suppose that $ \nu \neq \delta_1$ and that there exists $k \geq 1$ such that $ \nu_k >0 $ and $ \mu_{(k)} \neq \delta_1$.
Let $\vphi (\mathbf{t})$ be the flux of the parking process on $\mathbf{t}$ i.e.~the number of exiting cars. Given a vertex $x$ of $\mathbf{t}$, we sometimes denote by $ \vphi_x (\mathbf{t})$ the flux at vertex $x$ of $\mathbf{t}$, i.e. the outgoing flux of the parking process on $\Top(\mathbf{t}, x)$ the subtree of the descendants of $x$ in $\mathbf{t}$. 
To characterize the location of the phase transition, we introduce the size-biased distribution $\overline{ \nu} (k) = k \nu_{ k}$ for $k \geq1$ and the quantities 
\[ \E_{\overline{\nu}} [m] := \sum_{k = 0}^{\infty} k \nu_k \, m_{(k)},  \quad \E_{\nu} [m] := \sum_{k = 0}^{\infty} \nu_k \, m_{(k)} \quad\text{ and } \quad \E_{\nu} [\sigma^2 + m^2 -m] := \sum_{k = 0}^{\infty}\nu_k  \left( \sigma_{(k)}^2 + m_{(k)}^2 -m_{(k)}\right).\] 

\begin{theorem}[Phase transition for parking]\label{DegDep}
We assume $\E_{\overline{\nu}} [m] \leq 1$. The parking process on Galton--Watson tree undergoes a phase transition which depends on the sign of the quantity 
\begin{equation}\label{tTheta}
\Theta := (1- \E_{\overline{\nu}} [m])^2 -\Sigma^2 \E_{\nu} [\sigma^2 + m^2 -m].
\end{equation}
More precisely, we have three regimes classified as follows: 
 \begin{center}\label{tab}
\begin{tabular}{|c||c|c|c|}
  \hline
 & subcritical  & critical  & supercritical  \\
 & $\Theta >0$ & $\Theta =0$ &$\Theta <0$ \\
  \hline \hline
  $\vphi(\cT_n)$ as $n \to \infty$ & converges in law & $\underset{n \to \infty}{\xrightarrow{(\P)}} \infty$ but is $o(n)$ & $\sim cn$ with $c>0$ \\
  \hline \hline
$\Sigma^2 \E[\vphi(\cT)] + \E_{\overline{\nu}} [m]-1$ &$ - \sqrt{\Theta}$ & $0$ & $\infty$ \\ 
  \hline \hline
$\P \left( \diameter \text{ is parked in } \cT \right)$ &$\E_{\nu} [m]$&$\E_{\nu} [m]$&$\E_{\nu} [m]-c$ \\
\hline
\end{tabular}
 \end{center}
\end{theorem}
\bigskip

As an example of application of Theorem \ref{DegDep}, if the cars can arrive only on the leaves with law $\mu_{(0)}$, then the phase transition occurs for $\Theta_{\text{leaf}} = 1 - \Sigma^2 \nu_0 ( \sigma_{(0)}^2 + m_{(0)}^2 -m_{(0)})$, where $m_{(0)}$ and $\sigma_{(0)}^2$ are respectively the expectation and the variance of the number of arrivals at a leaf. However, if the cars arrive with the same density but on every vertex i.e.~if the distribution of the car arrivals is $\mu = \nu_0 \mu_{(0)} + (1 - \nu_{0} ) \delta_0$ and does not depend on the degree of the vertex, then $\Theta_{\text{unif}} = (1-\nu_0 m_{(0)})^2  - \Sigma^2 \nu_0 ( \sigma_{(0)}^2 + m_{(0)}^2 -m_{(0)}) \leq \Theta_{\text{leaf}}$. This means that with the same density of cars, the parking can be subcritical if the cars arrive only on the leaves but supercritical if the cars arrive uniformly on every vertex. \\
A natural assumption for the parking process to be subcritical is that $\E_{\nu} [m] \leq 1$, so that there are typically fewer cars than parking spots. The assumption $\E_{\overline{\nu}} [m] \leq 1$ may sound unnatural but comes for the fact that the number of children of the vertices in a ``typical" branch have size-biased law $\overline{\nu}$ \cite{GP19}. Indeed, the parking process on Kesten's tree $\cT_{\infty}$ is supercritical when $\E_{\overline{\nu}} [m] \geq 1$, as we will see in Section \ref{MF}. As a consequence of this phase characterization, we can deduce that if $\E_{\overline{\nu}} [m] \leq 1$ and $\Theta >0$, the parking process is subcritical and therefore $\E_{\nu} [m] \leq 1$, and conversly, if $\E_{\overline{\nu}} [m] \leq 1$ and $\E_{\nu} [m] > 1$, then the parking process is supercritical and hence $\Theta <0$. However these implications are not derived by easy algebraic manipulations.\\
As we said above, Theorem \ref{DegDep} generalizes the result of \cite{CH19} and its proof follows the same lines. It is presented in Section \ref{PT}.

\paragraph{Sharpness and large deviations.} Our main contribution in this paper consists in showing that the phase transition established in Theorem \ref{DegDep} is sharp. More precisely, we shall reinforce the first line in the table of Theorem \ref{DegDep} by proving large deviations: 
 \begin{theorem}[Large deviations for the flux]\label{GrandDev}
 Let $\eps >0$. In the supercritical regime i.e.~if $\Theta <0$, there exists $\delta >0$, and $n_0 \geq 0$ such that for all $n \geq n_0$, 
\[ \P\left( \left| \vphi (\cT_n) - cn\right| \geq \eps n \right) \leq e^{-\delta n},\]
where $c>0$ is as in Theorem \ref{DegDep}.
In the subcritical regime, there exists $\delta >0$, and $n_0 \geq 0$ such that for all $n \geq n_0$, 
\[ \P\left( \left| \vphi (\cT) \right| \geq \eps n \right) \leq e^{-\delta n}.\]
\end{theorem}
Notice that the second item of Theorem \ref{GrandDev} applies to the unconditioned tree $\cT$. It holds also for $\cT_n$ after changing the constants since $\P (|\cT| = n)$ has polynomial probability. Our proof of Theorem \ref{GrandDev} is very different in the supercritical and subcritical cases. In the supercritical case, the large deviations will be established by showing first large deviations for the fringe subtree distribution of $\cT_n$ in Section \ref{SecFringe}. This may be a result of independent interest which complements the law of large numbers and the Central Limit Theorem of Aldous \cite{A91} and Janson \cite{J16, J12}. In the subcritical case, we adopt a very different analytic point of view. Following \cite{GP19} the flux at the root of $\cT$ satisfies a recursive distributional equation which turns into an analytic equation on its generating function $z \mapsto W(z)$. However, the equation has a singularity at $z=1$ and $W(1)=1$. By employing Newton--Puiseux expansion we are able to resolve this singularity and prove that in the subcritical case $z \mapsto W(z)$ has radius of convergence strictly larger than $1$. This is the object of Section \ref{SecLD}.

\paragraph{Offcritical geometry.} We will give an application of this exponential decay for the flux to the size of the connected components after the parking procedure, that is the clusters of occupied parking spots in $\cT_n$ in the subcritical and supercritical phases. We notice that this geometry shares many similarities with the size of the connected components of the Erd\H{o}s-Renyi random graph: only logarithmic clusters in the subcritical case and a giant component in the supercritical phase. This is actually not a mere coincidence and in forthcoming works we shall exhibit a strong link between parking on random trees and random graph processes \cite{C23, CC21}.

\begin{corollary}[Offcritical geometry]\label{CorGeom}Let $|C_{\max}(n)|$ be the size of the largest parked connected component in $\cT_n$, and $|C_2(n)|$ be the size second largest connected component. Then, 
\[
\begin{array}{ll} 
\mbox{\text{ (supercritical $\Theta <0$) }} \qquad & \dfrac{|C_{\max}(n)|}{n} \underset{n \to \infty}{\xrightarrow{(\P)}} C \quad \mbox{\text{ and}} \quad  \P (|C_2(n)|\geq A \ln (n) ) \underset{n \to \infty}{\xrightarrow{}} 0,  \\
& \\
\mbox{\text{ (subcritical $\Theta >0$) }} \qquad  &\P ( |C_{\max}(n)| \geq A \ln (n)) \underset{n \to \infty}{\xrightarrow{}} 0, 
 \end{array}
\]
where $C\in (0,1)$ and $A>0$ are constants that depend on the laws $ \nu $ and $ \mu_{(k)}$ for $k \geq 0$.
\end{corollary}

\paragraph{Acknowledgments.} I would like to thank warmly Nicolas Curien for precious suggestions and corrections. I am also very grateful to Olivier Hénard for many interesting discussions.

\section{Phase transition and fringe subtrees}\label{PT}

In this section, we generalize the phase transition result of \cite{CH19} to our case, that is when the distribution of the car arrivals depends on the degree of the vertex. The strategy of proof is very similar and we will only highlight the necessary adaptations. The crux is the adaptation of \cite[Proposition 1]{CH19} into Proposition \ref{meanflux}.

\subsection{The mean flux and the probability that the root is parked in $\cT$}\label{MF}

We first obtain the expected outgoing flux of the unconditioned Galton--Watson tree using a differential equation. To this purpose, we let the cars arrive according to random times $A_x$ uniform in $[0,1]$ independently on each vertex $x$ of the tree $\cT$. More precisely, conditionally on $\cT$, we define a family $(A_x)_{x \in \cT}$ of i.i.d.~random variables with law $\operatorname{Unif}[0,1]$ independently of car arrivals $(L_x)_{x \in \cT}$. We denote by $\vphi (\cT , t) = \vphi (t)$ the outgoing flux on the root of $\cT$ after the parking procedure with car arrivals $L^{(t)}_x = \one_{A_x \leq t} L_x$ on each vertex $x \in \cT$ conditionally on $\cT$.  Note that conditionally on $\cT$, the car arrivals $(L_x^{(t)})_{x \in \cT}$ are independent with law $\mu_{(k)}^{(t)} = (1-t)\delta_0 + t\mu_{(k)}$ if $x$ has $k \geq 0$ children.
 
 \begin{prop}[Phase transition for the mean flux]\label{meanflux} For $t \in [0,1]$, we denote by $\Phi (t) = \E [\vphi (\cT , t) ]$ the mean flux at the root of $\cT$ with car arrivals with law $\mu^{(t)}_{(k)}$. Let $t_{\max}$ be the smallest solution to $ (1- \E_{\overline{\nu}} [m]t)^2 = \Sigma^2 \E_{\nu} [\sigma^2 + m^2 -m]t$ in $[0,1]$ (set $t_{\max} = + \infty$ if there is no such solution). Then, for $t \in [0,1]$
\begin{equation}\label{fluxdeg}
\Phi (t)=
\left\lbrace 
\begin{array}{ccc}
\dfrac{(1-  \E_{\overline{\nu}} [m]t ) - \sqrt{ (1- \E_{\overline{\nu}} [m]t)^2- \Sigma^2 \E_{\nu} [\sigma^2 + m^2 -m]t}}{\Sigma^2} & \mbox{if} & t \leq t_{\max}\\
+ \infty & \mbox{if} & t > t_{\max}.
\end{array}\right.
\end{equation}
 \end{prop}
 
 \begin{proof} 
 We use the same notation as in \cite[Proposition 1]{CH19}: if $x$ is a vextex of $\cT$, we denote by $I^{x} (s)$ the number of cars that arrived at time $s$ on the vertex $x$ which contribute to $\vphi(t)$, i.e. those that did not manage to park at their arrival time $s \leq t$.
For $ t \in [0,1]$, we have
\begin{eqnarray*}
\Phi (t) = \E \left[ \sum_{x \in \cT}I^{x} (A_x) \mathds{1}_{0 \leq A_x \leq t}\right] 
= \E \left[ \sum_{x \in \cT}\sum_{k = 0}^{\infty}I^{x} (A_x) \mathds{1}_{0 \leq A_x \leq t} \mathds{1}_{ \{ x \text{ has degree } k \}}\right]
\end{eqnarray*}

Recall that the degree of a vertex $x$ in a tree $\mathbf{t}$ is a function of $\Top( x, \mathbf{t})$ the subtree of the descendants of $x$. We use the many-to-one formula (see e.g.~\cite[Formula 3]{CH19}) and integrate on $s = A_x$ to obtain 
\begin{eqnarray*}
 \Phi (t) = \int_{0}^t  \mathrm{d}s \sum_{k =0}^{\infty} \sum_{h =0}^{\infty}  \E \left[ I(s,h) \mathds{1}_{ \{ S_h \text{ has degree } k \} }\right],
\end{eqnarray*}  
where $I(s,h)$ is obtained as follows: First recall the construction of Kesten tree $\cT_{\infty}$. Consider a semi-infinite line $S_0, S_1, \ldots,$ rooted at $S_0$, called the \emph{spine}, and graft independently on each $S_i$ a random number $Y - 1$ of independent Galton–Watson trees where $Y \sim \overline{\nu}$, and consider a random uniform ordering of the children of $S_i$. 
Here, we define a tree $\cT(h)$, for $h \geq 0$, by considering only a finite line $S_0, S_1, \ldots, S_h$ and grafting  independently on each $S_i$ a random number $Y - 1$ of independent Galton–Watson trees where $Y \sim \overline{\nu}$ for $ 0 \leq i < h$, and consider a random uniform ordering of the children of $S_i$ and furthermore $X$ independent copies of $\cT$ on $S_h$ where $X \sim \nu$ (see Figure \ref{Tk}). This tree is decorated by letting cars arrive with law $\mu_{(l)}^{(s)}$ at each vertex of degree $l$ independently, except on the vertex $S_h$ where we put an independent number of cars distributed as $\mu_{(k)}$ (instead of $ \mu_{(k)}^{(s)}$) when $S_h$ has degree $k$. Then $I(s,h)$ is the number of those cars arriving on $S_h$ that do not manage to park after all other cars of $\cT(h)$ have parked.

\begin{figure}[h!]
\centering
\includegraphics[width=10.8cm]{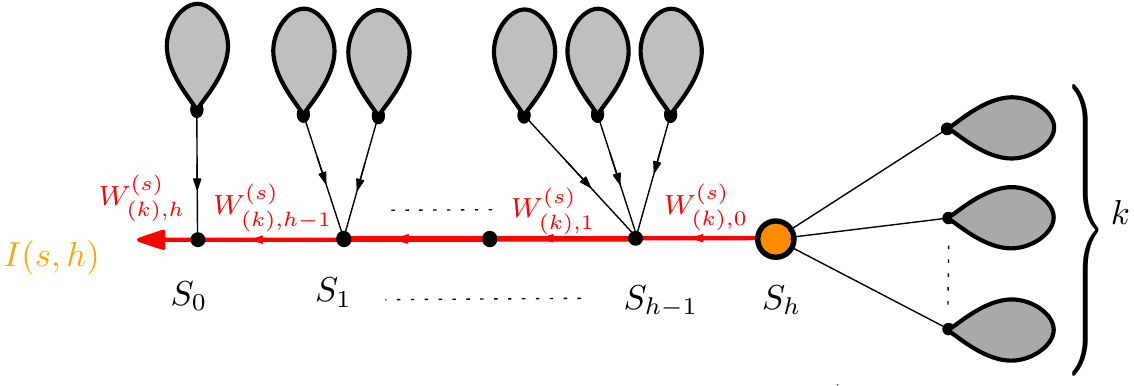} 
\caption{The tree $\cT(h)$ conditioned by $\{ S_h \text{ has degree } k \} $. }\label{Tk}
\end{figure}

To compute $\E \left[ I(s,h) \mathds{1}_{S_h \text{ has degree } k}\right]$, we use the fact that at time $s$ and on the event $\left\{ S_h \text{ has degree } k \right\} $, the outgoing flux from the vertex $S_{h-i}$ (before parking the cars arriving on $S_h$)  is a random walk $W_{(k),i}^{(s)}$ of length $h$ with  i.i.d.~increments of law $Z^{(s)} - 1$ with
\[ Z^{(s)} = \sum_{i = 1}^{Y-1} F_i^{(s)}  + P^{(s)}_{(Y)}\]
where $Y \sim \overline{\nu}$, the $F_i^{(s)}$ are copies of $F^{(s)} \sim \vphi(\cT , s )$ for $i \geq 0$, the $P^{(s)}_{(k)}$ have law $\mu_{(k)}^{(s)}$ for $k \geq 0$ and all the variables are independent. The starting point of the random walk $W_{(k),0}^{(s)}$ is distributed as the sum of  $k$ independent copies of $\vphi (\cT , s)$ minus $1$. We define $T_{-i}^{(s)} $ to be the first hitting time of $-i$ by the walk $W_{(k)}^{(s)}$ for $i \leq 0$ and we write  $ \mathbb{P}_x$ for the law of $W_{(k)}^{(s)}$ started at $x$ and $ \mathbb{E}_x$ for the corresponding expectation, for $x \in \mathbb{Z}$.
Summing over $h$, we obtain in the same way as in \cite[Proposition 1]{CH19}
\begin{eqnarray*}
\sum_{h =0}^{\infty} \E \left[ I(s,h)\mathds{1}_{ S_h \text{ has degree } k}\right] &=& \nu_k \left( \frac{1}{2} ( \sigma_{(k)}^2 + m_{(k)}^2 -m_{(k)}) + k m_{(k)} \Phi (s)\right)  \E_{0}\left[T_{-1}^{(s)} \right]. \\
\end{eqnarray*}
On the one hand, if $\E[ Z^{(s)} -1 ] \geq 0$, the random walk $W_{(k)}^{(s)}$ has a positive drift, hence $ \E_{0}\left[T_{-1}^{(s)} \right] = \infty$. On the other hand, if  $\E[ Z^{(s)} -1 ] < 0$, the random walk $W_{(k)}^{(s)}$ has a strictly negative drift and by Wald equality
\[\E_{0}\left[T_{-1}^{(s)} \right] = \frac{1}{\E[ 1-Z^{(s)} ] } = \frac{1}{1 -\E_{\overline{\nu}} [m]s-  \Sigma^2 \Phi (s)}.  \]

Summing over $k \geq 0$, we obtain $\Phi (0) = 0$, and for all $t \leq t_{c}$, 
\[ \Phi (t) =  \int_{0}^t \mathrm{d}s  \frac{\frac{1}{2}\E_{\nu} [\sigma^2 + m^2 -m]   + \E_{\overline{\nu}} [m] \Phi (s) }{1 - s \E_{\overline{\nu}} [m] - \Sigma^2 \Phi (s)},\]

where $t_c= \inf \{t \geq 0 : 1- \E_{\overline{\nu}} [m]t-  \Sigma^2 \Phi (t) <0 \}$. We can easly check that the function defined  on the right-hand side of \eqref{fluxdeg} satisfies this equation. It remains to check that both functions ``blow up" at the same time $t_{\max} = t_{c}$. This is done in the proof of Proposition 1 in \cite{CH19} using monotone and dominated convergence.
\end{proof}

When $ \mathbb{E}_{ \overline{ \nu}} [m ] \leq 1$, then $t \mapsto (1- \E_{\overline{\nu}} [m]t)^2 - \Sigma^2 \E_{\nu} [\sigma^2 + m^2 -m]t$ is decreasing over $[0,1]$. Hence, we obtain the following phase characterization for parking : when $t_{\max} <1$ then $\Theta < 0 $ (supercritical regime), when $t_{\max} =1$ then $ \Theta = 0$ (critical regime)  and when $t_{\max} >1$ then $\Theta >0$  (subcritical regime).
Using the same technique, we now control the probability that the root of an unconditioned Galton--Watson tree is parked as in \cite[Proposition 2]{CH19}.
 \begin{prop}[Probability that the root is parked]\label{probaroot} We have
 \begin{equation}
\P \left( \diameter \text{ is parked in } \cT \right)=
\left\lbrace 
\begin{array}{ccc}
  \E_{\nu} [m]& \mbox{\text{ if }} & \Theta \geq 0, \\
<    \E_{\nu} [m] & \mbox{ if } & \Theta < 0.
\end{array}\right.
\end{equation}
 \end{prop}

 \begin{proof}We proceed as in the previous proposition and let the cars arrive according to random times $A_x$ independently for each vertex $x$  on the tree $\cT$. Let $p_t = \P( \diameter \text{ contains a car in }(\cT , t))$. Then, 
 \begin{eqnarray*}
 p_t = \int_0^{t} \mathrm{d}s \sum_{h \geq 0} \P (P(s,h))
 \end{eqnarray*}
 where $P(s,h)$ is the event that  in the labeled tree $\cT(h)$, one car arriving on the vertex $S_h$ at time $s$ goes down the spine and manages to park on the empty root $\diameter$. Moreover, conditionally on $\left\{ S_h \text{ has } k\text{ children} \right\}$, the event $P(s,h)$ is $ \cup_{i=1}^{L_{(k)}} \left\{ T_{-i}^{(s)} =h \right\}$ under $\P_{W_{(k),0}^{(s)}}$ where $L_{(k)} \sim \mu_{(k)}$. Thus, 
\begin{eqnarray*}
\sum_{h=0}^{\infty} \P (P(s,h) \cap  \left\{ S_h \text{ has } k\text{ children} \right\} ) &=& \nu_k \E \left[ \sum_{i=1}^{L_{(k)}} \P_{W_0^{k,(s)}} (T_{-i}^{(s)} < \infty )\right] \\
&=&  \nu_k \E \left[ \sum_{i=1}^{L_{(k)}} \P_{0} (T_{-1}^{(s)} < \infty )^{W_0^{k,(s)} + i }\right] \\
&=& \left\{\begin{array}{ccc} 
  \nu_{k} m_{(k)} & \mbox{\text{ if }} & s \leq t_{\max} \\
 <   \nu_{k} m_{(k)} & \mbox{ if } & s > t_{\max}. \end{array}\right.
\end{eqnarray*}
 Summing over $k$ and integrating over $s$, we get the desired result.
 \end{proof}

\subsection{Fringe subtrees and weak law of large numbers for the flux} \label{SecFringe}

We now want to show a convergence result for the flux of the conditioned Galton--Watson tree $\cT_n$.
To this end, a useful tool will be the \L ukasiewicz walk (in the depth-first order) of the Galton--Watson tree (see \cite{LG}), decorated with the car arrivals on each vertex. Therefore in the rest of  the paper, we consider $(S,L)$ a random process where $S$ is a random walk with starting point $S_0 = 0$ and  i.i.d.~increments of law $\P( S_1 = k) = \nu_{k+1}$ for $ k \geq - 1$, and conditionally on $(S)$, the $(L_i)_{i \geq 0}$ are independent of law $\mu_{(S_{i+1}- S_{i}+1)}$. Note that the $(L_i)_{i \geq 0}$ are i.i.d.~and $\E[L_0] = \E_{\nu} [m]$. 
We also define $T$ the first hitting time of $-1$ by the walk $S$. Then the law of $(S_i, L_i)_{0 \leq i \leq n}$ conditionally on $\left\{ T = n \right\}$ is the law of the \L ukasiewick walk of a Galton--Watson tree $\cT_n$ conditioned to have $n$ vertices ``decorated" with the car arrivals on each vertex. Therefore, we couple $(S,L)$ with $\cT_n$ and $\cT$ so that $(S,L)_{0 \leq i \leq T}$ (resp.~$(S,L)_{0 \leq i \leq n}$) is the \L ukasiewick walk of $\cT$ (resp.~$\cT_n$) decorated with the car arrivals conditionally on $T$ (resp.~$\left\{ T = n \right\}$).

\begin{prop}[Law of large number for the flux]\label{LGN} The flux at the root of $\cT_n$ satisfies 
\[  \frac{\varphi(\cT_n)}{n} \underset{n \to \infty}{\xrightarrow{(\P)}} \E_{\nu} [m ]  - \P \left( \diameter \text{ is parked in } \cT \right). \]
\end{prop}

\begin{proof}
By conservation of cars,  the total number of cars arriving on $\cT_n$ is 
\begin{eqnarray*}
\left(  \sum_{x \in \cT_n } L_x \right) =  \varphi(\cT_n)  + \sum_{x \in \cT_n } \one_{x \text{ is parked}}.
\end{eqnarray*}
We first prove that the proportion of arriving cars per vertex in $\cT_n$ converges in probability towards $\E_{\nu}[m]$ as $n$ goes to $\infty$. Recall that we defined a random walk $S$ with i.i.d~increments of law $\nu$ together with $L$ the car arrival ``decoration" and $T$ the first hitting time of $-1$, so that conditionally on $\{T = n\}$ and on $\cT_n$, the car arrival on $v_k$, the $k$th vertex of $\cT_n$ in the depth-first order, is $ L_{v_k} = L_{k-1}$  and has law $\mu_{(S_{k} - S_{k-1} +1)}$ for $1 \leq k \leq n$. 
Then, for all $\eps >0$,
\begin{eqnarray}\label{DevTotCars}
\P\left( \left|  \sum_{x \in \cT_n } L_x - \E_{\nu}[m]n  \right| \geq \eps n \right) &\leq& \dfrac{\P \left( \left| \sum_{k=0}^{n-1}  (L_{k} - \E_{\nu}[m] ) \right| \geq \eps n \right)}{\P ( T = n)}.
\end{eqnarray}
Since the $L_k$ are i.i.d.~and their common law has by assumption \eqref{Hexp} an exponential tail, we can use large deviations inequalities and bound the above numerator by $e^{- \delta n }$ for some $\delta >0$. Moreover we recall the classical asymptotic 
\begin{equation} \label{PSize} \P (T = n)  \sim C n^{-3/2}
 \end{equation}
for some $C>0$ as $n$ goes to $+ \infty$ (at least along values for which $\P (T =n) >0$). Hence the probability on the left-hand side converges to $0$ as $n$ goes to $\infty$ and does so exponentially fast. 

Then we observe that the degree of a vertex $x$ in a given tree $\mathbf{t}$ is a function of $\Top( x, \mathbf{t})$, the subtree of the descendants of $x$. Therefore we can use the theorem of Janson \cite[Theorem 1.3, formula (1.11)]{J16} which states that the fringe subtree distribution of a conditioned Galton--Watson tree $\sum_{x \in \cT_n}\delta_{\Top( x, \cT_n)} /n$ converges in probability to the $\nu$-Galton--Watson measure. Adding car arrivals decoration on each vertex, this implies that
\[ \frac{1}{n} \sum_{x \in \cT_n } \one_{x \text{ is parked}} \underset{n \to \infty}{\xrightarrow{(\P)}} \P \left( \diameter \text{ is parked in } \cT \right). \]
The desired result follows. 
\end{proof}

We have seen above that the total number of cars arriving on $\cT_n$ concentrates around its (unconditioned) expectation $\E_{\nu}[m]$. 
To obtain large deviations for the flux (at least in the case of supercritical parking process), we also need a large deviations result for the fringe subtree distribution. We establish such a result in a more general context, which concerns not only the subtree of the descents of the vertices, but a more general local neighborhood.
Let $\mathbf{t}$ be a plane tree, $x \in \mathbf{t}$ be a vertex of $\mathbf{t}$ and $k \geq 0$. We define $H_k(\mathbf{t},x) = \Top (\mathbf{t}, x_k)$  where $x_k$ is the $k$th ancestor of $x$ if there is one. Otherwise, we just say $H_k(\mathbf{t},x) = \lozenge$.
When $\mathbf{t} = \cT_n$ and $x = \mathfrak{u}_n$ is a uniform vertex of $\cT_n$, Aldous \cite{A91} (see also Stufler \cite{St19}) has proved that $\cT_n$, seen from the vertex $\mathfrak{u}_n$ converges in distribution (for the local topology) towards an infinite plane tree with almost surely one spine $\cT^*$ called the random sin-tree: Consider $(u_{k})_{k \geq 0 }$ a semi-infinite path, ``pointed" at $u_0$, such that $u_{k+1}$ is the ancestor of $u_k$ for all $k\geq 0$. Then graft independently on $u_0$ a random number $X$ of independent Galton–Watson trees where $X \sim \nu$ and on each $u_k$ for $k \geq 1$ a random number $Y - 1$ of independent Galton–Watson trees where $Y \sim \overline{\nu}$, and consider a random uniform ordering of the children of $u_k$. Note that for all $k \geq 0$, the subtree $\Top( \cT^*, u_k) = H_k(\cT^*,u_0)$ has law $\cT(k)$ (see Figure \ref{Tk}).
Aldous' sin-tree $\cT^*$ is closely related to the Kesten tree $\cT_{\infty}$: whereas $\cT_{\infty}$ describes the local limit of $\cT_n$ near the root vertex, $\cT^*$ describes its local limit near a ``typical" vertex.

\begin{prop}[Large deviations for the fringe subtrees]\label{thmloc} Let $\mathbf{t}$ be a (fixed) plane finite tree and $k \geq 0$ an integer such that the height of $\mathbf{t}$ is at least $k$. For every $\eps >0$, there exists $\delta >0$ and $n_0 \geq 0$ such that for all $n \geq n_0$, 
\[ \P\left( \left| \frac{1}{n}\sum_{x \in \cT_{n}} \mathds{1}_{H_k(\cT_n,x) = \mathbf{t}}  - \P ( H_k(\cT^*, u_0)= \mathbf{t}) \right| \geq \eps \right) \leq e^{-\delta n},\]
\end{prop}

\begin{proof} We shall prove the result using the \L ukasiewicz walk of $\cT_n$ which is almost a random walk and for which large deviations for density of patterns is easy to see.
Let $L^{\mathbf{t}} = (L^{\mathbf{t}}_i)_{0 \leq i \leq |\mathbf{t}|}$ be the \L ukasiewicz walk of $\mathbf{t}$  and $(v_i)_{0 \leq i \leq |\mathbf{t}|-1}$ be the vertices of $\mathbf{t}$ in the depth-first order, so that $L^{\mathbf{t}}_{i+1} - L^{\mathbf{t}}_i +1$ is the number of children of $v_i$ in $\mathbf{t}$ for $0 \leq i \leq |\mathbf{t}| -1$. We denote by $M$ the maximum of the increments of $L^{ \mathbf{t}}$. 

We extend the walk $(S)$ as a bi-infinite walk by setting $S_j = 1$ for $j<0$. For $j  \in \Z$, we write $(S_i^{(j)} = S_{i+j} - S_{j})_{i \in \Z}$ for the walk shifted at time $j$. We claim that there exists a function $f_{\mathbf{t}}$ defined over bi-infinite paths and taking values in $\{0,1 \}$ such that  conditionally on $\{ T = n \}$, 
\begin{align}\label{FoncF}
H_k (\cT_n, v_j) = \mathbf{t} \qquad  \Longleftrightarrow  \qquad f_{\mathbf{t}}(S^{(j)}) = 1,
\end{align}
where $(v_k)_{0 \leq k \leq n-1}$ are the vertices of $\cT_n$ listed in the depth-first order.
Specifically, $f_{\mathbf{t}}$ is defined as follows: considering $W \in \Z^{\Z}$, we first define $\tau_0 = 0$, then $\tau_{1} = \sup \{  j < \tau_{0}, W_j \leq W_{\tau_0}  \}$ if it exists ($ - \infty$ otherwise), and so on up to $\tau_{k} = \sup \{  j < \tau_{k-1} , W_j \leq W_{\tau_{k-1}}  \} $, so that if $\tau_{k} <\infty$, the $\tau_{i}$'s correspond to the location of the ancestors of the vertex ``0". We then set $f_{\mathbf{t}}(W) =1$ if and only if  $\tau_{k} \geq - | \mathbf{t}|$ and $(W_{j}^{(\tau_{k})})_{0 \leq j \leq |\mathbf{t}|} = L^{\mathbf{t}}$. In particular, the value of $f_{\mathbf{t}} (W)$ only depends on $(W_{k})_{-| \mathbf{t}| \leq k \leq | \mathbf{t}|}$. More than that, if $\tau_{k} \geq - \mathbf{t}$ and $W_{k+1} - W_k > M$ for some $\tau_{k}\leq k \leq  \tau_{k} + | \mathbf{t}|$, then $f(W) = 0$ (and changing the values of $W_k$ for $ k \leq \tau_{k} $ or  $k \leq \tau_{k} + | \mathbf{t}|$ does not change the value of $f_{ \mathbf{t}}$). Therefore we consider the random walk $(\tS_k)_{k \in \Z}$ such that $\tS_k = S_k$ for $ k \leq 0$ and $\tS_{k+1} - \tS_k = (S_{k+1} - S_k) \wedge (M+1)$, and the corresponding shifted walk $\tS^{(j)}$.

Using \eqref{FoncF}, we obtain
\begin{align*}
&\P\left( \left| \frac{1}{n}\sum_{x \in \cT_{n}} \mathds{1}_{H_k(\cT_n,x) = \mathbf{t}}  - \P (H_k(\cT^*,u_0) = \mathbf{t})  \right| \geq \eps \right) \\
&= \P\left( \left| \frac{1}{n}\sum_{j =0}^{n-1} f_{\mathbf{t}}(\tS^{(j)}) - \P (H_k(\cT^*,u_0) = \mathbf{t})\right| \geq \eps \middle| T = n\right) \\
&\leq \dfrac{\P\left( \left| \frac{1}{n}\sum_{j =0}^{n-1} f_{\mathbf{t}}(\tS^{(j)})- \P (H_k(\cT^*,u_0)= \mathbf{t})\right| \geq \eps \right)}{\P(T = n)}.
\end{align*}

Apart from the first $ 0 \leq j \leq | \mathbf{t}|$ values, the function $ j \mapsto f(\tS^{(j)}) = f((\tS^{(j)}_{k})_{-| \mathbf{t}| \leq k \leq | \mathbf{t}|})$ is a function of the underlying Markov chain $((\tS^{(j)}_{k})_{| \mathbf{t}| \leq k \leq | \mathbf{t}|})_{j \geq 0}$. 
Furthermore, by Aldous \cite[Proposition 10]{A91} we can see that $ \P (H_{k}(\mathcal{T}^{*},u_0) = \mathbf{t} )$ is the probability that $f_{ \mathbf{t}} (\tS) = 1$ under the stationary distribution of the Markov chain $((\tS^{(j)})_{-| \mathbf{t}| \leq k \leq | \mathbf{t}|})_{j \geq | \mathbf{t}|}$ which is simply the law of the two-sided random walk with increments $ \nu_{ (\cdot +1)\wedge (M+1)}$. Since $\nu_{ (\cdot +1)\wedge (M+1)}$ has finite support, this is a Markov chain with finite state space, and  Sanov's theorem \cite[Section 6.2]{DZ93} implies large deviations for the numerator, which is bounded above by $e^{- \delta n}$ for some $ \delta > 0$ for $n$ large enough.
Using \eqref{PSize}, we get the desired result.
\end{proof}

\proof[Sketch of Proof of Theorem \ref{DegDep}] 
The second line of Theorem \ref{DegDep} can be easily derived from Proposition \ref{meanflux}.
Moreover Proposition \ref{probaroot} already gives us the third line of the table of Theorem \ref{DegDep} and together with Proposition \ref{LGN}, we obtain that the flux is linear when $ \Theta <0$ and sublinear when $ \Theta \geq 0$. There only remains to check that  the flux converges in law in the subcritical case and diverges in the critical case. The proof is an adapation of \cite[Section 4]{CH19} and in particular \cite[Lemma 3]{CH19}, which shows that no car coming from far away in $ \mathcal{T}_n$ contributes to $ \vphi ( \mathcal{T}_n)$.
\endproof

\section{Large deviations for the parking process}\label{SecLD}

\subsection{Supercritical parking process}

We are now able to prove Theorem \ref{GrandDev} in the supercritical case, i.e~when $\Theta <0$. 

\begin{proof}[Proof of Theorem \ref{GrandDev}, supercritical case] As in the proof of Proposition \ref{LGN}, by conservation of cars, we have
\[ \vphi (\cT_n)  = \sum_{x \in \cT_n } L_x - \sum_{x \in \cT_n} \mathds{1}_{x \text{ is parked}}.\]
Hence, 
\begin{equation*}
\P\left( \left| \vphi (\cT_n) - cn \right| \geq  \eps n \right)  \leq \P\left( \left| \sum_{x \in \cT_n } L_x - \E_{\nu}[m]n \right| \geq \frac{\eps}{2} n \right) 
+ \P\left( \left| \sum_{x \in \cT_n} \mathds{1}_{x\text{ is parked}} - \left( \E_{\nu}[m] - c\right)n \right| \geq \frac{\eps}{2} n  \right).
\end{equation*}

Using the bound \eqref{DevTotCars} for the total number of cars, the first term of the right-hand has exponential decay. We now want to have large deviations for the number of occupied parking spots. Let $M \geq 0$ be ``large enough".
Then
\begin{eqnarray*}
\P\left( \left| \sum_{x \in \cT_n}  \mathds{1}_{x \text{ is parked}} - \left( \E_{\nu}[m] - c\right) n  \right| \geq \frac{\eps}{2}n  \right) \leq \P\left(  \left| \sum_{x \in \cT_n}  \mathds{1}_{|\Top(\cT_n, x)| > M } \right| \geq  \frac{\eps}{4}n  \right)  \\
+ \P\left( \left| \sum_{x \in \cT_n} \mathds{1}_{x \text{ is parked}}  \mathds{1}_{|\Top(\cT_n, x) | \leq M} - \left( \E_{\nu}[m] - c \right) n \right| \geq  \frac{\eps}{4} n  \right).
\end{eqnarray*}

For the first term, we use Proposition \ref{thmloc} with $k=0$: choosing $M \geq 0$ such that $\P(|\Top(\cT^*,u_0)| \leq M ) = \P (|\cT |\leq M) \geq 1 - \varepsilon /8$ and applying the large deviations result of  Proposition \ref{thmloc} on the finitely many trees $\mathbf{t}$ such that $ | \mathbf{t}| \leq M$, we get $\delta_2>0$ and $n_2 \geq 0$ such that for all $n \geq n_2$, 
\begin{eqnarray*}
\P\left(  \sum_{x \in \cT_n}  \mathds{1}_{|\Top(\cT_n, x)| \leq M } <  \left(1 - \frac{\eps}{4}\right)n  \right) & \leq& \P\left( \left|  \sum_{x \in \cT_n}  \mathds{1}_{|\Top(\cT_n, x)| \leq M }  - n \P (|\cT | \leq M) \right| >  \frac{\eps}{8} n  \right)
  \leq e^{-\delta_2 n }.
\end{eqnarray*}

By Theorem \ref{DegDep}, the probability $\P ( \diameter \text{ is parked in } \cT)$ is $\E_{\nu}[m] - c$. We therefore can choose $M \geq 0$ such that 
$$ \P ( \diameter \text{ is parked in } \cT,  |\cT| \leq M  \text{ and }  \forall x \in \mathcal{T}, L_{x} \leq M) \geq \E_{\nu}[m] - c - \varepsilon /8.$$
Then using an easy extension of Proposition \ref{thmloc} with car arrivals decoration, we obtain
\begin{align*} &\P\left( \left| \sum_{k = 0}^n \mathds{1}_{x \text{ is parked}}  \mathds{1}_{|\Top(\cT_n, x) | \leq M} - \left( \E_{\nu}[m] - c \right) n \right| \geq  \frac{\eps}{4} n  \right) \\ 
&\leq  \P\left( \left| \sum_{k = 0}^n \mathds{1}_{x \text{ is parked}}  \mathds{1}_{|\Top(\cT_n, x) | \leq M}  -   \P ( \diameter \text{ is parked in } \cT \text{ and } |\cT| \geq M ) n \right| \geq  \frac{\eps}{8} n  \right) \leq  e^{- \delta_3 n }.
  \end{align*}
for some $ \delta_3 >0$ and for $n$ large enough. We get the desired result by combining theses inequalities. 
\end{proof}

\subsection{Subcritical parking process}

We now prove Theorem \ref{GrandDev} in the subcritical case, i.e~when $\Theta >0$. Our strategy of proof is very different from the supercritical case and requires analytical and geometric arguments.
Let $X$ be the number of cars that visit the root of $\cT$ and $W$ its generating function, i.e.
\begin{equation*} \label{FuncGen}
W(z) = \sum_{k=0}^{+ \infty} z^k \P ( k\text{ cars visit the root})
\end{equation*}
Since $\vphi (\cT) = (X - 1)_{+} = \sup (X-1,0) $, it suffices to show that $X$ has an exponential tail. More precisely, we will show that $W$ is a convergent series and has radius of convergence strictly larger than $1$ so that the probability $\P( k\text{ cars visit the root})$ has exponential decay.
Using the branching property at the root of $\cT$ we see that $X$ is a solution to the following recursive distributional equation: 
\begin{equation}\label{DE}\tag{DE}
X  \stackrel{(d)}{=} \sum_{i=1}^{N} (X_i - 1)_{+} + P_{N},
\end{equation}
where $N \sim \nu$, the $P_k$ have law $ \mu_{(k)}$ for $k \geq 0$,  the $X_i$ for $i \geq 0$ are i.i.d.~copies of the variable $X$ and all variables on the right-hand side are independent. Therefore, $W$ satisfies the following equation at least in terms of formal power series:
\begin{equation}\label{EQ}\tag{EQ}
W(z) = \sum_{k\geq 0 } \nu_k A_k (z) \left( \frac{1}{z}(W(z) - p_0) + p_0  \right)^k,
\end{equation}
where $A_k$ is the generating function of $P_k \sim \mu_{(k)}$ and  $p_0 = \P \left( \diameter \text{ is parked in } \cT \right)$. This equation on $W$ was used in \cite{GP19} and \cite{CG19} in the case of Poisson Galton--Watson trees and geometric or Poisson arrivals of cars. In these cases, they gave an explicit solution for $W$ in the subcritical case. Notice in passing that \eqref{EQ} or equivalently \eqref{DE} characterizes the law of $X$ and in particular, the quantity $p_0$ which appears in \eqref{EQ} is determined by \eqref{EQ}. In the subcritical case however, we already proved $p_0 = 1- \E_{\nu}[m]$ in Proposition \ref{probaroot} by probabilistic means. Plugging this into \eqref{EQ} we deduce that $z \mapsto W(z)$ is solution of $F(z-1,W(z)-1)=0$ where
\begin{eqnarray*}
F(x,y) = \sum_{k\geq 0 } \nu_k A_k (1+x) \left( \frac{1}{1+x}(y+1 - p_0) + p_0  \right)^k - y -1 = \sum_{i,j \geq 0} a_{i,j} x^{i} y^{j},
\end{eqnarray*}
for some family $(a_{i,j})_{i,j \geq 0}$, which is analytic around $(0,0)$ in both variables by the assumption \eqref{Hexp}.
When $z = 0$ and $W(0) = p_0$, we have
\[ \partial_{y} ((1+x)  \times F) (-1  , p_0 -1))= \sum_{k \geq 0} k\nu_{k} A_{k} (0) p_0^k  \neq 0,\]
and so by the analytic implicit function theorem, $z \mapsto W(z)$ is locally well defined around $z=0$ as the unique solution of $F(z-1,w(z)-1)\times z = 0$ and $w(0) =p_0$.
In fact, since $W$ is a generating function, its radius of convergence is at least $1$ and we have $W(1) = 1$. The problem is that $z = 1$ and $W(1)=1$ may be a singularity for the equation: indeed, we have 
\begin{align*} 
\partial_y F(x,y)&= -1+ \sum_{k \geq 0} k \nu_{k} \frac{ A_{k} (1+x)}{1+x} \left( \frac{y+1-p_0}{1+x} +p_0  \right)^{k-1},   \\
\partial_x F(x,y)&= \sum_{k \geq 0} \nu_{k} A_{k}' (1+x) \left( \frac{y+1-p_0}{1+x} +p_0  \right)^{k} - k \nu_{k} A_{k} (1+x)\frac{y+1 - p_0 }{(1+x)^2} \left( \frac{y+1-p_0}{1+x} +p_0  \right)^{k-1}.
\end{align*}
Since $A_k$ is the generating function of $ \mu_{(k)}$, we have $A_k(1) = 1$ and $A_k'(1) = m_{(k)}$. Using the fact that $p_0 = 1- \E_{\nu}[m]$ (Proposition \ref{probaroot}), we obtain  $\partial_x F(0,0) = 0$ and $\partial_y F(0,0) = 0$, so one cannot use the implicit function theorem at this point to extend $W(z)$. In the subcritical case we will resolve this singularity and prove that $F(z,w)$ has two analytic branches at $(0,0)$.  

\begin{lemma} \label{lem:branch} The origin is a double point of the section defined by $F (x,y) = 0$: that means that in a neighborhood of $(0,0)$, the equation $F(x,y(x)) = 0$ has two analytic branches.
\end{lemma}

\begin{figure}[!h]
 \begin{center}
 \includegraphics[width=6.8cm]{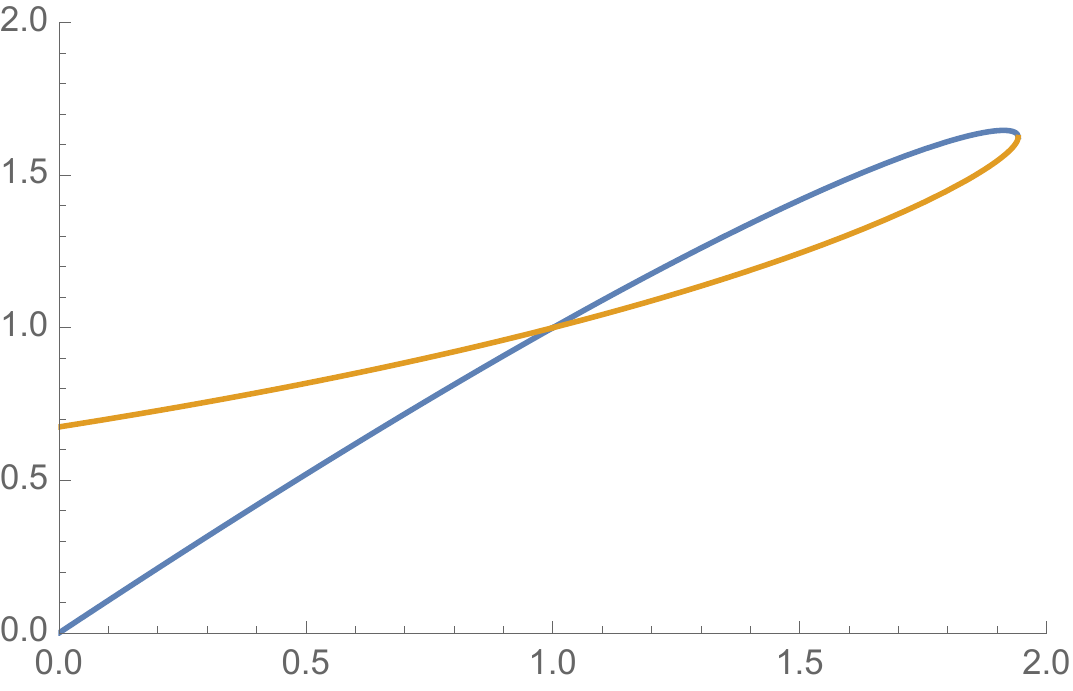}
 \caption{The curve defined by $F(x+1, y+1) = 0$ when the offspring distribution is geometric and the car arrivals are i.i.d~with Poisson distribution of parameter $ \alpha =0.325 \leq \alpha_c = \sqrt{2}- 1$ (see \cite{CG19}). Notice the two analytic branches around $(1,1)$. The orange one is the generating function $W$. } \label{fig:analytic}
 \end{center}
 \end{figure}

\proof We apply Newton's method to determine the Newton--Puiseux expansion at the point $(0,0)$ \cite[p498-500]{FS}. The Newton--Puiseux expansion of a solution $y$ of $F(x,y(x)) = 0$ around $(0,0)$ shows that any solution can be expressed as a locally convergent series of the form
\[ y(x) = \sum_{k \geq k_0} c_k x^{k/d}\]
where $k_0 \in \mathbb{Z}$ and $d$ in an positive integer. Inductively we will show that  $d =1$, $ k_0 \geq 0$ and that we have only two choices for the sequence $(c_k)_{k \geq k_0}$ in our case and so that $F$ has two solutions $y_{\pm}$ which are analytic around $0$. We look for a solution to $ F(x, y(x)) = 0$ of the form $y(x) = c x^a + o (x^a)$ with $c \neq 0$. We expand $F$ up to order $2$ and we obtain 
\begin{eqnarray*}
0 &=&a_{2,0} x^2 +  a_{0,2} c^2 x^{2a} + a_{1,1} c x^{1+a} + o (x^{2a \wedge 2}) \\
&=&\left( \E_{\nu} [m^2 + \sigma^2 - m ] + \E_{\overline{\nu}}[m]^2 ( \Sigma^2 - 2) - 2 \E_{\overline{\nu}}[m]\right) \frac{x^2}{2} + \frac{\Sigma^2 c^2}{2} x^{2a} \\
&~&+ c \left( \E_{\overline{\nu}}[m](1 - \Sigma^2) -1 \right)x^{1+a} + o (x^{2a \wedge 2}) 
\end{eqnarray*}

Since the equation has to be identically satisfied, the main asymptotic should be $0$. This can only happen if two or more of the exponents in  $\{2, 1+a, 2a\}$ coincide and the coefficients of the corresponding monomial in $F$ are zero. We obtain here $a=1$ and $c$ satisfies a quadratic equation that has two different real solutions which are 
\begin{equation}
c_{\pm} = \frac{- \left( \E_{\overline{\nu}}[m](1 - \Sigma^2) -1 \right) \pm \sqrt{\Theta} }{\Sigma^2}
\end{equation}

We choose one of the two solutions for $c_1$ and suppose that we have a solution of the form $y(x) = c_1 x +\cdots + c_{k-1} x^{k-1} + o (x^{k-1})$ and  look for a solution of the form $y(x) = c_1 x +\dots + c_{k-1} x^{k-1} + c_k x^{a}  + o (x^a)$ where $c_k \neq 0$, $k \geq 2$ and $a > k-1$. Expanding $F$ up to order $k+1$, we obtain 
\begin{align*}
0 &= \sum_{1 \leq i, j \leq k+1} a_{i,j} x^i \left(c _1 x +\dots + c_{k-1} x^{k-1} + c_k x^{a}  + o (x^a) \right)^j + o (x^{k+1}) \\
&= c_k \left( c_1 \frac{\Sigma^2}{2} +  \left( \E_{\overline{\nu}}[m](1 - \Sigma^2) -1 \right) \right)x^{1+a} +  x^{k+1} \sum_{1 \leq i, j \leq k+1} a_{i,j} \sum_{\substack{1 \leq l_1, \ldots, l_j \leq k \\ \sum l_n = k+1-i} } c_{l_1} \ldots c_{l_j} + o(x^{(1+a) \wedge (k+1)})
\end{align*}

Since $c_k \neq 0$, then the coefficient of the term $x^{1+a}$ is $c_k c_1 \Sigma^2/ 2+ c_k \left( \E_{\overline{\nu}}[m](1 - \Sigma^2) -1 \right) = \pm \sqrt{\Theta} c_k \neq 0$. In the generic situation, the coefficient of the term $x^{k+1}$ is non-zero and since the main asymptotic of $F$ should be $0$, then $a = k$ and 
\begin{eqnarray*}
c_k = \frac{1}{\sqrt{\Theta}} \sum_{2 \leq i, j \leq k+1} a_{i,j} \sum_{\substack{1 \leq l_1, \ldots, l_j \leq k \\ \sum l_n = k+1-i} } c_{l_1} \ldots c_{l_j} \neq 0 .
\end{eqnarray*}
Otherwise i.e.~when the coefficient of the term $x^{k+1}$ is zero, we expand $F$ up to the next orders until we find a non-zero coefficient. Note that since $a >1$, the term $x^{1+a}$ can only cancel out with a term where $x$ appears at an integer power. 
Therefore, we obtain a Puiseux expansion of $y$, which has integer powers, which concludes the proof. 
\endproof

\begin{proof}[Proof of Theorem \ref{GrandDev}, subcritical case] 
By Lemma \ref{lem:branch}, the equation $F(z-1, w(z) -1)=0$ has two analytic branches around $(z,w) = (1,1)$ (Figure \ref{fig:analytic}). One of these branches coincides with the generating function of the number of cars which visit the root $W$ defined in \eqref{FuncGen} in a neighborhood of $1^{-}$. Therefore, $z = 1$ is not a singularity of $W$. Moreover since $W$ is a generating function, its power expansion around $z=0$ has non-negative coefficients. Hence, by Pringhsheim's theorem \cite[Theorem IV.6 p240]{FS}, the radius of convergence of $W$ around $0$ is greater than $1$ and the desired result follows.
\end{proof}

\section{Application to the size of the connected components}

In this section we want to study the size of the connected components in the final configuration i.e.~the clusters of vertices that contain a car after the parking procedure and prove Corollary \ref{CorGeom}. We say that $x \in \mathbf{t}$ is free (resp. parked) if it contains (resp. does not contain) a car after parking.
\subsection{Supercritical Case : Giant component}

In this section we suppose that we are in the supercritical case i.e.~$ \Theta <0$ and prove Corollary \ref{CorGeom} in this phase.
We first prove that the second largest connected component is of size $O(\ln(n))$. This can be easily deduced from the following proposition.
\begin{prop}\label{LGsuper}When $ \Theta <0$, there exists $A_0\geq 0$, such that for all $A > A_0$,
\begin{eqnarray*} 
\P ( \exists x \in \cT_n \text{ s.t. }  | \Top( \cT_n , x)| \geq A \ln (n) \text{ and $x$ is a free spot}) \xrightarrow[n\to\infty]{}0  
\end{eqnarray*}
\end{prop}
\begin{proof} We let $v_k$ be the $k$th vertex in the depth-first exploration of $\cT_n$ for $0 \leq k \leq n-1$. We use the fact that conditionally on $|\Top (\cT_n, v_k) |= N$, the tree $\Top (\cT_n , v_k)$ has law $\cT_N$: for every fixed tree $ \mathbf{t}$,
\begin{eqnarray} \label{EqCond} 
\P \Big( \Top (\cT_n , v_k) = \mathbf{t} \Big| |\Top (\cT_n, v_k) |= N \Big) = \P ( \mathcal{T}_N = \mathbf{t}) .
\end{eqnarray}
We first bound the probability in the proposition by the expectation of the number of such vertices $x$.Therefore, the probability in the proposition satisfies 
\begin{align*} 
&\P ( \exists x \in \cT_n  \text{ s.t. } |\Top( \cT_n , x)| \geq A \ln (n) \text{ and } x \text{ is a free spot}) \\
 &\underset{}{\leq} \sum_{k=1}^n \E \left[ \mathds{1}_{|\Top( \cT_n , v_k)| \geq A \ln (n) \text{ and $v_k$ is a free spot}}\right] \\
 &\underset{ \eqref{EqCond}}{\leq}  \sum_{k=1}^n \sum_{s \geq A \ln (n)} \E \Big[ \mathds{1}_{v_k\text{ is a free spot}} \Big| |\Top( \cT_n , v_k)| = s \Big] \P (|\Top( \cT_n , v_k)| =s )  \\
 &\underset{ }{\leq}  n \underset{s \geq A \ln(n)}{\sup}\P ( \diameter \text{ is a free spot in }\cT_s) \\
 &\underset{\text{Thm } \ref{GrandDev}}{\leq} n \times e^{- \delta A \ln(n)}, 
\end{align*}
where $\delta$ is independent of $A$ (and $n$). Therefore, if $A \geq 2 / \delta$, this quantity converges to $0$ as $n$ goes to $\infty$.
\end{proof}

Recall that $|C_{\max}(n)|$ is the size of the largest parked connected component of $\cT_n$. Recall Aldous' sin-tree $ \mathcal{T}^* $ from Section \ref{SecFringe} with spine $\{ u_0, u_1, \ldots \}$. \begin{prop} We have the following convergence 
\begin{align*}
\dfrac{|C_{\max}(n)|}{n} \underset{n \to \infty}{\xrightarrow{(\P)}} \P \left( \forall k \geq 0, u_k \text{ is parked in } \cT^{*} \right),
\end{align*}
where the probability on the righthand side is computed by imagining that we perform the parking on $\cT^*$ (rooted at infinity)  with the same rules for car arrivals as for $ \mathcal{T}$.
\end{prop}

\begin{proof} We chose $A > A_0$ and we work on the complement event of that of Proposition \ref{LGsuper} i.e.~on $ \mathcal{E}_n = \{ \forall x \in \cT_n \text{ s.t. }  | \Top( \cT_n , x)| \geq A \ln (n), \text{$x$ is parked}\}$. Then, when $n$ is large enough, the root is parked and its parked component contains all vertices $x$ such that $| \Top( \cT_n , x)| \geq A \ln (n)$ and is therefore with high probability the only parked component of size larger than $A \ln (n)$. Hence, we can decompose the vertices of $\cT_n$ according to whether they have an ancestor which is a free parking spot and how far this ancestor is: on $ \mathcal{E}_n$, when $n$ is large enough and with high probability, 
\begin{eqnarray*} 
|C_{\max}(n) | 
 &=& n -  |\{ x \in \cT_n \text{ s.t. } x \text{ has an ancestor at distance  $ < A\ln(n)$, which is a free spot} \}| \\ 
&&- |\{ x \in \cT_n \text{ s.t. }x \text{ has an ancestor at distance  $ \geq A\ln(n)$, which is a free spot} \}|.
\end{eqnarray*}

Since $A \geq A_0$ as defined in Proposition \ref{LGsuper}, then on $ \mathcal{E}_n$ i.e.~with high probability, 
\begin{align*}
|\{ x \in \cT_n, x \text{ has an ancestor at distance  $ \geq A\ln(n)$, which is a free spot } \}|  = 0.
\end{align*}

Moreover, using \cite[Theorem 5.2]{St19} and its easy extension that takes care of the car decoration, since $\ln(n) = o (\sqrt{n})$, we deduce that
\begin{align*}
\Big| \frac{1}{n}|\{ x \in \cT_n \text{ s.t. }x \text{ has an ancestor at distance  $ < A\ln(n)$, which is a free spot } \}| \\
- \mathbb{P} \left( \forall 0 \leq k \leq A \ln (n) , u_k \text{ is parked in } \cT^{*} \right)\Big|  \underset{n \to \infty}{\xrightarrow{(\P)}} 0.
\end{align*}
As the probability $ \mathbb{P} ( \mathcal{E}_n)$ converges to  $0 $ and  $\mathbb{P} \left( \forall 0 \leq k \leq A \ln (n) , u_k \text{ is parked in } \cT^{*} \right)$ converges to $\mathbb{P} \left( \forall  k \geq 0 , u_k \text{ is parked in } \cT^{*} \right)$ as $n$ goes to $\infty$, we get the desired result. 
\end{proof}

\subsection{Subcritical parking}

In this section we suppose that we are in the subcritical case i.e.~$ \Theta >0$ and prove Corollary \ref{CorGeom} in this case.

\begin{proof}[Proof of Corollary \ref{CorGeom}]
The proof is based on the sprinkling method which consists in adding cars while staying in the subcritical phase. Since $\Theta > 0$, there exists $\eps >0$ such that 
\begin{align*}
\Theta' = \Theta + \eps^2 - 2 \eps (1- \mathbb{E}_{ \overline{ \nu}}[m] - \Sigma^2 \E_{ \nu} [m ] ) > 0,
\end{align*}
This means that the parking process on $\cT_n$ or $\cT$ with offspring distribution $\nu$ and car arrivals with distribution $\tilde{\mu}_{(k)}^{\eps}$ such that $\tilde{\mu}_{(k)}^{\eps}(j) = (1- \eps)\mu_{(k)}(j) + \eps \mu_{(k)} (j-1)$ (that is we add a car with probability $\eps$ on each vertex independently) is still subcritical. We denote by $\tvphi$ the corresponding flux. Recall that $\vphi_x( \mathbf{t})$ is the flux at vertex $x$ i.e.~on $\Top( \mathbf{t}, x)$.
Imagine that in $ \mathcal{T}_n$ with arrivals $ \mu$, we have a large parked component $ C$ of size larger than $A \ln(n)$. If we further let cars arrive on each vertex with probability $ \varepsilon$, then by the law of large numbers, the root $x$ of $C$ gets a flux $\tvphi_x( \mathcal{T}_n) \geq A \varepsilon \ln (n) /2$ with probability at least $1/2$ when $n$ is large enough. Therefore, for $n$ large enough, 
\begin{align*}
\P \left( |C_{\max}(n)| \geq A \ln (n))\right)  &\leq 2 \P \left( \exists x \in \cT_n, \tvphi_{x} (\cT_n ) \geq \frac{\eps A }{2} \ln (n )\right) \\
&\leq 2 \sum_{M =1}^{n}  \sum_{k =1}^n \P \Big( \tvphi_{v_k} (\cT_n )\geq \frac{\eps A }{2} \ln (n ) \Big| |\Top(\cT_n, v_k)| = M\Big) \P \left( |\Top(\cT_n, v_k)| = M \right).
\end{align*}
Now we use the conditional law \eqref{EqCond}, the asymptotic \eqref{PSize} and Theorem \ref{GrandDev} so that
\begin{eqnarray*}
\P \left( |C_{\max}(n)| \geq A \ln (n)\right)  &\underset{ \eqref{EqCond}}{\leq}& 2\sum_{M = 1}^{n} \sum_{k =1}^n  \P \left( |\Top(\cT_n, v_k)| = M\right) \frac{ \P \left(\tvphi(\cT) \geq \frac{\eps A }{2} \ln (n )\right)}{ \P \left( |\cT| = M\right) } \\
&\underset{ \eqref{PSize}}{\leq} & 2 C \sum_{M \geq 1} \sum_{k =1}^n  \P \left( |\Top(\cT_n, v_k)| = M \right) \times n^{3/2} \times \P \left(\tvphi_{} (\cT )\geq \frac{\eps A }{2} \ln (n ) \right)\\
&\underset{\text{Thm } \ref{GrandDev}}{\leq} &2 C n^{5/2} \times n^{- \delta \eps A/2},
\end{eqnarray*}
for some constant $C>0$ and some $\delta >0$. Therefore, choosing $A > 2/(\delta \eps)  + 5/2$, we obtain the desired result.
\end{proof}


\end{document}